\documentclass{article}
\usepackage[T1]{fontenc}
\usepackage[utf8]{inputenc}
\usepackage[style=alphabetic,sorting=nyt,giveninits=true]{biblatex}
\usepackage{geometry}
\usepackage{parskip}
\usepackage{amsmath}
\usepackage{amssymb}
\usepackage{amsthm}
\usepackage{hyperref}
\usepackage[capitalize,nameinlink,noabbrev]{cleveref}
\usepackage{mathtools}
\usepackage{commath}
\usepackage{bm}
\usepackage{enumitem}
\usepackage{tcolorbox}
\usepackage{microtype}
\usepackage[hang,flushmargin]{footmisc}
\usepackage{thmtools} 
\usepackage{thm-restate}
\usepackage{tikz}

\tikzset{
  graphvertex/.style={
    circle,
    draw=black,
    fill=black,
    inner sep=0pt,
    minimum size=5pt
  },
  mygraph/.style={
    every node/.style={graphvertex},
    label distance=0.15cm,
    line width=0.25mm
  }
}

\makeatletter
\g@addto@macro\bfseries{\boldmath}
\makeatother

\makeatletter
\def\thmCite{\@ifnextchar[{\@with}{\@without}}
\def\@with[#1]#2{{\normalfont\cite[#1]{#2}\;}}
\def\@without#1{{\normalfont\cite{#1}\;}}
\makeatother

\ExplSyntaxOn
\DeclareExpandableDocumentCommand{\IfNoValueOrEmptyTF}{mmm}
 {
  \IfNoValueTF{#1}{#2}
   {
    \tl_if_empty:nTF {#1} {#2} {#3}
   }
 }
\ExplSyntaxOff

\newcommand*{\valOrBlank}[1]{\IfNoValueOrEmptyTF{#1}{}{#1}}

\NewDocumentCommand{\mySum}{e{_^}}{%
	\sum_{\mathclap{\valOrBlank{#1}}}^{\mathclap{\valOrBlank{#2}}}\hspace{0.025cm}%
}

\NewDocumentCommand{\myCap}{e{_^}}{%
	\bigcap_{\mathclap{\valOrBlank{#1}}}^{\mathclap{\valOrBlank{#2}}}\hspace{0.05cm}%
}

\tcbset{sharp corners, colback=white, colframe=black, boxrule=0.4pt,parbox=false}

\makeatletter
\newcommand*{\defeq}{\hspace{-0.04cm}\mathrel{\rlap{%
                     \raisebox{0.4ex}{$\scriptstyle\m@th\cdot$}}%
                     \raisebox{-0.05ex}{$\scriptstyle\m@th\cdot$}}%
				 =}
\makeatother

\begingroup
\makeatletter
\@for\theoremstyle:=definition,remark,plain\do{%
	\expandafter\g@addto@macro\csname th@\theoremstyle\endcsname{%
		\addtolength\thm@preskip\parskip
	}%
}
\endgroup

\let\oldproof\proof
\def\proof{\oldproof\unskip}

\DeclareMathOperator{\hgt}{height}
\DeclareMathOperator{\inord}{in}
\DeclareMathOperator{\scnum}{sc}
\DeclareMathOperator{\scstar}{sc\hspace{-0.02cm}\textsuperscript{$\ast$}\hspace{-0.0375cm}}
\DeclareMathOperator{\LF}{\mathsf{LF}}

\DeclarePairedDelimiter{\smallAbs}{|}{|}

\newtheorem{theorem}{Theorem}[section]
\newtheorem{lemma}[theorem]{Lemma}
\newtheorem{corollary}[theorem]{Corollary}
\newtheorem{proposition}[theorem]{Proposition}
\newtheorem{definition}[theorem]{Definition}

\newtheorem{notation}[theorem]{Notation}
\theoremstyle{definition}
\newtheorem{example}[theorem]{Example}

\theoremstyle{remark}
\newtheorem*{note}{Note}

\newcommand{\blockComment}[1]{}

\newcommand{\code}[1]{{\normalfont\ttfamily #1}}

\DeclareTextFontCommand{\emph}{\boldmath\bfseries}

\newcommand{\citelink}[2]{{\hypersetup{linkbordercolor={0 1 0}}\hyperlink{cite.\therefsection @#1}{#2}}}

\newcommand*{\citeMacaulay}{\citelink{M2}{\code{Macaulay2}}}

\AtBeginBibliography{\vspace*{8pt}}

\DeclareFieldFormat[article,book,incollection,misc,unpublished]{title}{\mkbibitalic{#1}}
\DeclareFieldFormat[article]{journaltitle}{#1\isdot}
\DeclareFieldFormat[article,incollection]{booktitle}{\normalfont{#1}}

\title{Binomial Edge Ideals of K\"{o}nig Type}
\author{David Williams}
\date{\vspace{-0.3cm}}

\begin{document}

\maketitle

\begin{abstract}
	\noindent We first characterise graphs with binomial edge ideals of K\"{o}nig type as those for which the path covering number is equal to a minor variant of the scattering number.

	\vspace{0.35\baselineskip}
	\noindent This enables us to apply known graph-theoretic results to immediately deduce that several classes of graphs have binomial edge ideals of K\"{o}nig type. In particular, we show this for cocomparability graphs, or weakly closed graphs in the language of Matsuda (\cite{matsudaWeaklyClosedGraphs2018}).

	\vspace{0.35\baselineskip}
	\noindent Along with work of LaClair and McCullough (\cite[Corollary 10.4]{laclairFpurityBinomialEdge2026}), this allows us to prove that an unmixed binomial edge ideal $\mathcal{J}_G$ is of K\"{o}nig type if and only if $G$ is weakly closed.

	\vspace{0.35\baselineskip}
	\noindent We then conjecture that AT-free graphs have binomial edge ideals of K\"{o}nig type.
\end{abstract}

\begin{tcolorbox}
	Unless specified otherwise, let
	\begin{equation*}
		R=k[x_1,\ldots,x_n,y_1,\ldots,y_n]
	\end{equation*}
	for some field $k$ (of arbitrary characteristic) and $n\geq 2$.

	Throughout this paper, all graphs will be finite, simple and undirected. For a graph $G$, we denote by $V(G)$ and $E(G)$ the sets of vertices and edges $\{i,j\}$ of $G$ respectively. Unless specified otherwise, all graphs will be non-empty and have at most $n$ vertices, with vertex set $\{1,\ldots,\smallAbs{V(G)}\}$.

	We denote by $P_m$, $C_m$ and $K_m$ the path, cycle and complete graphs on $m$ vertices respectively, and by $K_{a,b}$ the complete bipartite graph with partitions of size $a$ and $b$. Furthermore, for any $S\subseteq\{1,\ldots,n\}$, we denote by $K_S$ the complete graph with vertex set $S$.
\end{tcolorbox}

\section{Introduction}\label{sec:intro}

In 1931, K\"{o}nig\footnote{\noindent D\'{e}nes K\H{o}nig (1884--1944) was a Hungarian mathematician whose name is more properly written with \textit{\H{o}}. However, since much of the literature uses \textit{\"{o}}, we adopt it here for consistency.} proved in \cite{konigGraphokEsMatrixok1931} that, for any bipartite graph, the matching number\footnote{The maximum size $\smallAbs{S}$ of a set $S\subseteq E(G)$ such that no two edges in $S$ share an endpoint (a \textit{matching}), denoted by $\nu(G)$, or occasionally $m(G)$.} is equal to the vertex cover number\footnote{The minimum size $\smallAbs{S}$ of a set $S\subseteq V(G)$ such that every edge of $G$ has an endpoint in $S$ (a \textit{vertex cover}), denoted by $\tau(G)$.}. This became known as \textit{K\"{o}nig's Theorem}, or sometimes the \textit{K\"{o}nig-Egerv\'{a}ry Theorem}, since the result was proved independently (and in the more general case of weighted graphs) in the same year by Egerv\'{a}ry in \cite{egervaryMatrixokKombinatoriusTulajdonsagairol1931}. This motivated the study of the general class of graphs for which these invariants agree, most commonly called \textit{K\"{o}nig graphs}.

It was shown by Crupi, Rinaldo and Terai in \cite[Theorem 0.2]{crupiCohenMacaulayEdgeIdeal2011} that the Cohen-Macaulayness of (monomial) edge ideals of K\"{o}nig graphs is independent of the characteristic of the base field. Herzog, Hibi and Moradi sought to generalise the properties of K\"{o}nig graphs which afforded this characteristic independence to arbitrary graded ideals in a polynomial ring over a field. To this end, they introduced in \cite[Definition 1.1]{herzogGradedIdealsKonig2022} the class of \textit{(graded) ideals of K\"{o}nig type}.

They proved in \cite[Theorem 3.5(b)]{herzogGradedIdealsKonig2022} that the binomial edge ideal $\mathcal{J}_G$ of a graph $G$ is of K\"{o}nig type if and only if $G$ contains a linear forest (or semi-path) with $\hgt_R(\mathcal{J}_G)$ edges. Alternative combinatorial characterisations were given by LaClair in \cite[Lemma 5.3]{laclairInvariantsBinomialEdge2025} and the author in \cite[Theorem 2.4]{williamsLFCoversBinomialEdge2023}.

Furthermore, Herzog, Hibi and Moradi showed in \cite[Proposition 3.6]{herzogGradedIdealsKonig2022} that graphs for which every connected component is traceable (that is, containing a Hamiltonian path) have binomial edge ideals of K\"{o}nig type, and moreover that, if a binomial edge ideal $\mathcal{J}_G$ is unmixed, then $\mathcal{J}_G$ is of K\"{o}nig type if and only if every connected component of $G$ is traceable.

\pagebreak

\subsection*{Summary of Results}

The main result of this paper is a characterisation of binomial edge ideals of K\"{o}nig type in purely graph-theoretic terms, building on Herzog, Hibi and Moradi's characterisation \cite[Theorem 3.5]{herzogGradedIdealsKonig2022}. Just as edge ideals $I(G)$ of K\"{o}nig type can be characterised via the equality of the matching and vertex cover numbers of $G$, we show that binomial edge ideals $\mathcal{J}_G$ of K\"{o}nig type can be characterised via the equality of two other well-studied graph-theoretic invariants: the \textit{path cover number} $\pi(G)$ (see \cref{def:pi}), and a minor variant of the \textit{scattering number} $\scnum(G)$ (see \cref{def:sc}). Setting $\scstar(G)\defeq\max\{1,\scnum(G)\}$, \hyperref[prf:main]{we prove} the following:
\begin{restatable}{theorem}{mainThm}\label{thm:main}
	Let $G$ be a graph. Then $\mathcal{J}_G$ is of K\"{o}nig type if and only if
	\begin{equation*}
		\pi(G)=\scstar(G)
	\end{equation*}
\end{restatable}
This characterisation allows us to apply existing graph-theoretic results to show that \hyperref[sec:KTClasses]{several classes} of graphs have binomial edge ideals of K\"{o}nig type.

In particular, \hyperref[prf:WCKT]{we answer} a question of LaClair and McCullough (\cite[Question 10.5]{laclairFpurityBinomialEdge2026}) concerning weakly closed graphs (see \cref{def:WC}) in the affirmative (and without their unmixedness hypothesis):
\begin{restatable}{corollary}{wcKTCor}\label{cor:WCKT}
	For any weakly closed graph $G$, $\mathcal{J}_G$ is of K\"{o}nig type.
\end{restatable}

This question stemmed from \cite[Corollary 10.4]{laclairFpurityBinomialEdge2026}: that if $\mathcal{J}_G$ is unmixed of K\"{o}nig type, then $G$ is weakly closed. Combined with \cref{cor:WCKT}, this allows us \hyperref[prf:unmixed]{to deduce}
\begin{restatable}{corollary}{unmixedCor}\label{cor:unmixed}
	Let $G$ be a graph, and suppose that $\mathcal{J}_G$ is unmixed. Then $\mathcal{J}_G$ is of K\"{o}nig type if and only if $G$ is weakly closed.
\end{restatable}
This provides a useful alternative to Herzog, Hibi and Moradi's characterisation of unmixed binomial edge ideals of K\"{o}nig type via traceability (\cite[Proposition 3.6]{herzogGradedIdealsKonig2022}), since deciding whether a graph admits a Hamiltonian path is NP-complete in general (see, for example, \cite[p.~60]{gareyComputersIntractabilityGuide1979}), whereas weakly closed graphs are equivalent to cocomparability graphs (see \cref{def:cocomp}) via, for example, \cite[Theorem 2.9]{matsudaWeaklyClosedGraphs2018}, and these can be recognised in polynomial time (see \cite[Theorem 5.33]{golumbicAlgorithmicGraphTheory2004}\footnote{This theorem proves that \textit{comparability} graphs can be recognised in polynomial time, but these are, by definition, the complements of cocomparability graphs, and so their recognition problems are equivalent.}).

We conclude with a conjecture regarding AT-free graphs (see \cref{def:ATFree}):
\begin{restatable}{conjecture}{ATFreeConj}\label{conj:ATFreeKT}
	For any AT-free graph $G$, $\mathcal{J}_G$ is of K\"{o}nig type.
\end{restatable}
This conjecture is supported by \hyperref[MacaulayEvidence]{computational evidence} from \citeMacaulay{}, and the fact that weakly closed graphs form a subclass of AT-free graphs (see, for example, \cite[Theorem 2.9]{matsudaWeaklyClosedGraphs2018} and \cite{brandstadtGraphClassesSurvey1999}).

\subsection*{Outline of Paper}

\begin{itemize}
	\item \cref{sec:BEIs} consists of the definition and key properties of binomial edge ideals.
	\item \cref{sec:KT} includes the \hyperref[def:KT]{definition} of ideals of K\"{o}nig type, a \hyperref[thm:heightLFKT]{characterisation} in the case of binomial edge ideals, some \hyperref[subsec:KTExamples]{examples} and \hyperref[subsec:KTNonExamples]{non-examples}, and a brief \hyperref[subsec:UMKTOverview]{overview} of current results on unmixed binomial edge ideals of K\"{o}nig type.
	\item \cref{sec:mainResult} introduces the necessary \hyperref[subsec:PCSN]{graph-theoretic concepts} and contains \hyperref[prf:main]{the proof} of \cref{thm:main}.
	\item \cref{sec:KTClasses} applies \cref{thm:main} to prove \cref{cor:WCKT}, and also that several \hyperref[OtherKTClasses]{other classes} of graphs have binomial edge ideals of K\"{o}nig type via this same characterisation.
	\item \cref{sec:umKT} proves several results relating to unmixed binomial edge ideals of K\"{o}nig type, including \cref{cor:unmixed}, \hyperref[prop:almostNoUMKT]{and} that almost no graphs have unmixed binomial edge ideals of K\"{o}nig type.
	\item \cref{sec:ATFreeConj} concerns \cref{conj:ATFreeKT}, with \hyperref[MacaulayEvidence]{evidence} in its favour, and other related \hyperref[ATFreeDiscussion]{discussions}.
\end{itemize}

\newpage

\section{Binomial Edge Ideals}\label{sec:BEIs}

Binomial edge ideals were introduced in \cite{herzogBinomialEdgeIdeals2010}, and independently in \cite{ohtaniGraphsIdealsGenerated2011}. We begin with their definition (parts of this subsection appeared previously in \cite[Section 1]{williamsLFCoversBinomialEdge2023}).
\begin{definition}\label{def:BEI}
	Let $G$ be a graph, and set $f_{i,j}\defeq x_iy_j-x_jy_i$ for $1\leq i<j\leq n$.

	We define the \emph{binomial edge ideal} of $G$ as
	\begin{equation*}
		\mathcal{J}_G\defeq(f_{i,j}:\{i,j\}\in E(G))
	\end{equation*}
\end{definition}

A key fact about binomial edge ideals is the following:
\begin{theorem}\thmCite[Corollary 2.2]{herzogBinomialEdgeIdeals2010}\label{thm:BEIRadical}
	Binomial edge ideals are radical.
\end{theorem}

In fact, an explicit description of the primary decomposition can be given by purely combinatorial means.

\begin{notation}
	Let $G$ be a graph, and $S\subseteq V(G)$. Denote by $c_G(S)$ the number of connected components of $G\setminus S$ (we will just write $c(S)$ when there is no confusion). Then we set
	\begin{equation*}
		\mathcal{C}(G)\defeq\{S\subseteq V(G):\text{\normalfont$S=\varnothing$ or $c(S\setminus\{v\})<c(S)$ for all $v\in S$}\}
	\end{equation*}
	That is, when we ``add back'' any vertex in $S\in\mathcal{C}(G)$, it must reconnect some components of $G\setminus S$.
\end{notation}

\begin{notation}
	Let $G$ be a graph, and $S\subseteq V(G)$. Denote by $G_1,\ldots,G_{c(S)}$ the connected components of $G\setminus S$, and let $\tilde{G}_i$ be the complete graph with vertex set $V(G_i)$. Then we set
	\begin{equation*}
		P_S(G)\defeq(x_i,y_i:i\in S)+\mathcal{J}_{\tilde{G}_1}+\cdots+\mathcal{J}_{\tilde{G}_{c(S)}}
	\end{equation*}
\end{notation}

\begin{proposition}
	$\mathcal{J}_{K_S}$ is prime for any $S\subseteq\{1,\ldots,n\}$, and therefore $P_T(G)$ is prime for any graph $G$ and $T\subseteq V(G)$ also.
\end{proposition}

\begin{proof}
	This follows from \cite[Theorem 2.10]{brunsDeterminantalRings1988}.
\end{proof}

\begin{theorem}\thmCite[Corollary 3.9]{herzogBinomialEdgeIdeals2010}\label{thm:BEIPD}
	Let $G$ be a graph. Then
	\begin{equation*}
		\mathcal{J}_G=\myCap_{S\in\mathcal{C}(G)}P_S(G)
	\end{equation*}
	is the primary decomposition of $\mathcal{J}_G$.
\end{theorem}

This allows us to obtain a combinatorial formula for $\hgt_R(\mathcal{J}_G)$.

\begin{lemma}\thmCite[Lemma 3.1]{herzogBinomialEdgeIdeals2010}\label{lem:height}
	Let $G$ be a graph, and $S\in\mathcal{C}(G)$. Then
	\begin{equation*}
		\hgt_R(P_S(G))=\smallAbs{S}+\smallAbs{V(G)}-c(S)
	\end{equation*}
\end{lemma}

\begin{note}
	\cref{lem:height} is stated in \cite{herzogBinomialEdgeIdeals2010} with $n$ in place of $\smallAbs{V(G)}$, since they assume that $G$ has exactly $n$ vertices. However we only require that $G$ has \textit{at most} $n$ vertices, and since the additional indeterminates do not affect the height, our statements agree.
\end{note}

\begin{corollary}\label{cor:height}
	Let $G$ be a graph. Then
	\begin{equation*}
		\hgt_R(\mathcal{J}_G)=\min\{\smallAbs{S}+\smallAbs{V(G)}-c(S):S\in\mathcal{C}(G)\}
	\end{equation*}
\end{corollary}

\begin{proof}
	This is immediate from \cref{thm:BEIPD} and \cref{lem:height}.
\end{proof}

\begin{corollary}\label{cor:addHeight}
	Let $G_1$ and $G_2$ be graphs with disjoint vertex sets. Then
	\begin{equation*}
		\hgt_R(\mathcal{J}_{G_1\sqcup G_2})=\hgt_R(\mathcal{J}_{G_1})+\hgt_R(\mathcal{J}_{G_2})
	\end{equation*}
	and $\mathcal{J}_{G_1\sqcup G_2}$ is unmixed if and only if both $\mathcal{J}_{G_1}$ and $\mathcal{J}_{G_2}$ are unmixed. In particular, unmixedness can be checked componentwise on the connected components of $G$.
\end{corollary}

\begin{proof}
	The first claim is immediate from \cref{cor:height}, and the second from this and \cref{lem:height}.
\end{proof}

\pagebreak

\section{An Overview of Binomial Edge Ideals of K\"{o}nig Type}\label{sec:KT}

We recall the definition of ideals of K\"{o}nig type introduced by Herzog, Hibi, and Moradi in \cite{herzogGradedIdealsKonig2022}.

\begin{definition}\thmCite[Definition 1.1]{herzogGradedIdealsKonig2022}\label{def:KT}
	Let $k$ be a field, $R=k[x_1,\ldots,x_n]$ for some $n\geq1$, and $I$ a graded ideal of $R$ of height $h$. We say that $I$ is of \emph{K\"{o}nig type} if there exists a sequence of homogeneous polynomials $\bm{\mathrm{f}}=f_1,\ldots,f_h$ such that:
	\begin{enumerate}
		\item The $f_i$ form part of a minimal system of generators of $I$.
		\item There exists a monomial order $<$ on $R$ such that $\inord_<(f_1),\ldots,\inord_<(f_h)$ is a regular sequence.
	\end{enumerate}
	More precisely, $I$ is of K\"{o}nig type \emph{with respect to $\bm{\mathrm{f}}$ and $<$}.
\end{definition}

\subsection{Characterisation \& First Properties}

There is an elegant characterisation of binomial edge ideals of K\"{o}nig type again due to Herzog, Hibi and Moradi (\cite[Theorem 3.5(b)]{herzogGradedIdealsKonig2022}). We first introduce some terminology and notation.

\begin{definition}
	We say that a graph is a \emph{linear forest} (or \emph{semi-path}) if every connected component is a path. For a graph $G$, we say that a linear forest in $G$ is \emph{maximum} if no other linear forest in $G$ has a greater number of edges.
\end{definition}

\begin{notation}
	For a graph $G$, we denote by $\LF(G)$ the number of edges of a maximum linear forest in $G$.
\end{notation}

\begin{theorem}\thmCite[Theorem 3.5(b)]{herzogGradedIdealsKonig2022}\label{thm:heightLFKT}
	Let $G$ be a graph. Then $\mathcal{J}_G$ is of K\"{o}nig type if and only if
	\begin{equation*}
		\hgt_R(\mathcal{J}_G)=\LF(G)
	\end{equation*}
\end{theorem}

Alternative combinatorial characterisations of binomial edge ideals of K\"{o}nig type were given by LaClair in \cite[Lemma 5.3]{laclairInvariantsBinomialEdge2025} and the author in \cite[Theorem 2.4]{williamsLFCoversBinomialEdge2023}.

\begin{note}
	\cite[Corollary 3.4]{herzogGradedIdealsKonig2022} shows that under the standard lex ordering
	\begin{equation*}
		x_1>\cdots>x_n>y_1>\cdots>y_n
	\end{equation*}
	if a binomial edge ideal is of K\"{o}nig type then we may assume that $\bm{\mathrm{f}}$ is given by the $f_{i,j}\in R$ corresponding to the edges $\{i,j\}\in E(F)$ for any maximum linear forest $F$ in $G$. Then, after fixing the standard lex ordering, we do not need to be concerned with specifying ``K\"{o}nig type with respect to $\bm{\mathrm{f}}$'' for binomial edge ideals.
\end{note}

\begin{corollary}
	Let $G$ be a graph. Then the property of $\mathcal{J}_G$ being of K\"{o}nig type can be checked componentwise on the connected components of $G$.
\end{corollary}

\begin{proof}
	This follows immediately from \cref{cor:addHeight} and \cref{thm:heightLFKT}, since clearly
	\begin{equation*}
		\LF(G_1\sqcup G_2)=\LF(G_1)+\LF(G_2)
	\end{equation*}
	for graphs $G_1$ and $G_2$ with disjoint vertex sets.
\end{proof}

As noted in the \hyperref[sec:intro]{introduction}, ideals of K\"{o}nig type were introduced to try and generalise the properties of K\"{o}nig graphs which lead to the characteristic independence of the Cohen-Macaulay property of their edge ideals. This is achieved in the case of binomial edge ideals:

\begin{theorem}\thmCite[Corollary 3.8]{herzogGradedIdealsKonig2022}\label{thm:KTCMIndep}
	Let $G$ be a graph, and suppose that $\mathcal{J}_G$ is of K\"{o}nig type. Then the Cohen-Macaulayness of $R/\mathcal{J}_G$ does not depend on the characteristic of $k$.
\end{theorem}

\begin{note}
	It is conjectured that the Cohen-Macaulay property of \textit{any} binomial edge ideal is independent of the characteristic of $k$. This would follow, for example, from \cite[Conjecture 1.1]{bologniniCohenMacaulayBinomial2022}.
\end{note}

\pagebreak

\subsection{Examples}\label{subsec:KTExamples}

Many important classes of graphs have binomial edge ideals of K\"{o}nig type.

\begin{example}
	It is easily verified that path, cycle, and complete graphs have binomial edge ideals of K\"{o}nig type.
\end{example}

\begin{example}\label{example:traceable}
	Traceable graphs have binomial edge ideals of K\"{o}nig type (see, for example, \cite[Proposition 3.6]{herzogGradedIdealsKonig2022}). We will recover this result in \cref{prop:traceableKT}.
\end{example}

\begin{note}\label{note:almostAllKT}
	P\'{o}sa proves in \cite{posaHamiltonianCircuitsRandom1976} that almost all graphs are Hamiltonian, and so, in particular, they are traceable. Then \cref{example:traceable} tells us that almost all graphs have binomial edge ideals of K\"{o}nig type.
\end{note}

\begin{example}
	Forests have binomial edge ideals of K\"{o}nig type (see, for example, \cite[Theorem 5.9]{laclairInvariantsBinomialEdge2025} or \cite[Theorem 2.33]{williamsLFCoversBinomialEdge2023}). We will recover this result in \cref{prop:forestKT}.
\end{example}

\begin{example}
	Complete bipartite graphs and trivially perfect graphs (that is, graphs containing no induced $P_4$ or $C_4$) have binomial edge ideals of K\"{o}nig type by \cite[Proposition 2.21]{williamsLFCoversBinomialEdge2023} and \cite[Proposition 2.23]{williamsLFCoversBinomialEdge2023} respectively. We will generalise these results in \cref{thm:cocomKT}.
\end{example}

Further such classes will be demonstrated in \cref{sec:KTClasses}.

\subsection{Non-Examples}\label{subsec:KTNonExamples}

Not every graph has a binomial edge ideal of K\"{o}nig type.

\begin{example}\label{example:net}
	\cite[p.~313 and Figure 2]{herzogGradedIdealsKonig2022} shows that $\mathcal{J}_G$ is not of K\"{o}nig type for
	\begin{equation*}
		G=\quad
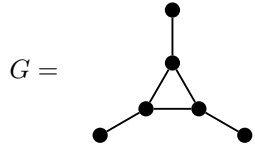
\begin{tikzpicture}[mygraph,x=0.7cm,y=0.7cm,mygraph,baseline={([yshift=-0.5ex]current bounding box.center)}]
   	   		\node(1) at (0,0) {};
    		\node(2) at ([shift=(-120:1)]1) {};
        	\node(3) at ([shift=(0:1)]2) {};
			\node(4) at ([shift=(90:1)]1) {};
			\node(5) at ([shift=(-150:1)]2) {};
        	\node(6) at ([shift=(-30:1)]3) {};
	  		\foreach \from/\to in {1/2,1/3,1/4,2/3,2/5,3/6}
				\draw[-] (\from) -- (\to);
    	\end{tikzpicture}
	\end{equation*}
	(the \textit{net}). It can be checked using \citeMacaulay{} that this example is minimal in the sense that it is the only graph on $6$ vertices or fewer with $\mathcal{J}_G$ not of K\"{o}nig type. It can also be checked that $\mathcal{J}_G$ is Cohen-Macaulay in this case, and so Cohen-Macaulayness is not enough to guarantee K\"{o}nig type.
\end{example}

\begin{example}\label{example:bipCE}
	Bipartite graphs do not necessarily have binomial edge ideals of K\"{o}nig type. \cite[Example 2.37]{williamsLFCoversBinomialEdge2023} shows that $\mathcal{J}_G$ is not of K\"{o}nig type for
	\begin{align*}
		G&=\quad
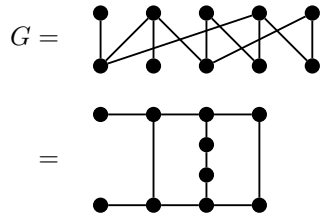
\begin{tikzpicture}[mygraph,x=0.7cm,y=0.7cm,baseline={([yshift=-0.5ex]current bounding box.center)}]
      		\node(1) at (0,0) {};
			\node(2) at (1,0) {};
			\node(3) at (2,0) {};
			\node(4) at (3,0) {};
			\node(5) at (4,0) {};
			\node(6) at (0,1) {};
			\node(7) at (1,1) {};
			\node(8) at (2,1) {};
			\node(9) at (3,1) {};
			\node(10) at (4,1) {};
	   		\foreach \from/\to in {1/6,1/7,1/9,2/7,3/7,3/8,3/10,4/8,4/9,5/9,5/10}
       	 		\draw[-] (\from) -- (\to);
    	\end{tikzpicture}\\[0.3cm]
		&=\quad\begin{tikzpicture}[mygraph,x=0.7cm,y=0.4cm,baseline={([yshift=-0.5ex]current bounding box.center)}]
	       	\node(1) at (1,3) {};
			\node(2) at (0,0) {};
			\node(3) at (2,0) {};
			\node(4) at (2,2) {};
			\node(5) at (3,3) {};
			\node(6) at (0,3) {};
			\node(7) at (1,0) {};
			\node(8) at (2,1) {};
			\node(9) at (2,3) {};
			\node(10) at (3,0) {};
		    \foreach \from/\to in {1/6,1/7,1/9,2/7,3/7,3/8,3/10,4/8,4/9,5/9,5/10}
	   	     	\draw[-] (\from) -- (\to);
		\end{tikzpicture}
	\end{align*}
	It can be checked using \citeMacaulay{} that this example is minimal in the sense that all bipartite graphs with fewer than $10$ vertices have binomial edge ideals of K\"{o}nig type, and there is only one other such example on $10$ vertices, but that contains an additional edge.
\end{example}

\begin{example}\label{example:KTDelVert}
	The property of $\mathcal{J}_G$ being of K\"{o}nig type is not closed under vertex deletion. For example, if we add a vertex $v$ to the net connecting two pendant vertices, so
	\begin{equation*}
		G=\quad
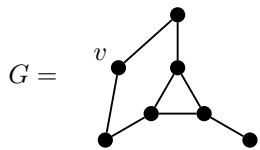
\begin{tikzpicture}[mygraph,x=0.7cm,y=0.7cm,baseline={([yshift=-0.5ex]current bounding box.center)}]
      		\node(1) at (0,0) {};
    		\node(2) at ([shift=(-120:1)]1) {};
        	\node(3) at ([shift=(0:1)]2) {};
			\node(4) at ([shift=(90:1)]1) {};
			\node(5) at ([shift=(-150:1)]2) {};
        	\node(6) at ([shift=(-30:1)]3) {};
			\node(7)[label={[label distance=2pt]150:$v$}] at ([shift=(180:{sqrt(5)/2})]1) {};
	  		\foreach \from/\to in {1/2,1/3,1/4,2/3,2/5,3/6,4/7,5/7}
				\draw[-] (\from) -- (\to);
    	\end{tikzpicture}
	\end{equation*}
	then $G$ is traceable, and so $\mathcal{J}_G$ is of K\"{o}nig type (see \cref{example:traceable}). However $G\setminus\{v\}$ is the net, which we have seen in \cref{example:net} does not have a binomial edge ideal of K\"{o}nig type.
\end{example}

\pagebreak

\subsection{Some Results on Unmixed Binomial Edge Ideals of K\"{o}nig Type}\label{subsec:UMKTOverview}

Herzog, Hibi and Moradi classified unmixed binomial edge ideals of K\"{o}nig type in terms of traceability:
\begin{proposition}\label{prop:HHMChar}
	Let $G$ be a graph, and suppose that $\mathcal{J}_G$ is unmixed. Then $\mathcal{J}_G$ is of K\"{o}nig type if and only if every connected component of $G$ is traceable.
\end{proposition}

\begin{proof}
	This is immediate from \cite[Proposition 3.6]{herzogGradedIdealsKonig2022}.
\end{proof}

We will give an \hyperref[cor:unmixed]{alternative classification} in \cref{sec:umKT}.

\begin{note}\label{note:C4}
	\cref{prop:HHMChar} does \textit{not} say that $\mathcal{J}_G$ is unmixed of K\"{o}nig type if and only if every connected component of $G$ is traceable. For example, label
	\begin{equation*}
		C_4=\quad
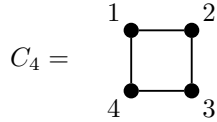
\begin{tikzpicture}[mygraph,x=0.8cm,y=0.8cm,label distance=2pt,baseline={([yshift=-0.5ex]current bounding box.center)}]
     		\node[label={135:$1$}](1) at (0,1) {};
      		\node[label={45:$2$}](2) at (1,1) {};
     		\node[label={-45:$3$}](3) at (1,0) {};
      		\node[label={-135:$4$}](4) at (0,0) {};
			\foreach \from/\to in {1/2,1/4,2/3,3/4}
				\draw[-] (\from) -- (\to);
    	\end{tikzpicture}
	\end{equation*}
	This is clearly traceable, but
	\begin{equation*}
		\mathcal{C}(C_4)=\{\varnothing,\{1,3\},\{2,4\}\}
	\end{equation*}
	which yields associated primes of heights $3$, $4$, and $4$, so $\mathcal{J}_{C_4}$ is not unmixed.
\end{note}

\begin{example}
	We saw in \cref{example:KTDelVert} that the property of $\mathcal{J}_G$ being of K\"{o}nig type is not closed under vertex deletion. Neither is the property of $\mathcal{J}_G$ being unmixed of K\"{o}nig type. For example, let
	\begin{equation*}
		G=\quad
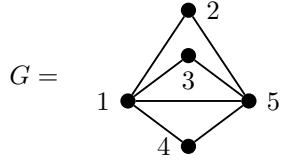
\begin{tikzpicture}[mygraph,x=0.8cm,y=0.6cm,label distance=2pt,baseline={([yshift=-0.5ex]current bounding box.center)}]
     		\node[label={180:$1$}](1) at (0,1) {};
      		\node[label={0:$2$}](2) at (1,3) {};
     		\node[label={-90:$3$}](3) at (1,2) {};
      		\node[label={180:$4$}](4) at (1,0) {};
			\node[label={0:$5$}](5) at (2,1) {};
			\foreach \from/\to in {1/2,1/3,1/4,1/5,2/5,3/5,4/5}
				\draw[-] (\from) -- (\to);
    	\end{tikzpicture}
	\end{equation*}
	that is, $K_{1,3}\vee K_1$, the cone over the \textit{claw}. This is traceable (for example, $2\to1\to3\to5\to4$), and
	\begin{equation*}
		\mathcal{C}(G)=\{\varnothing,\{1,5\}\}
	\end{equation*}
	which yields associated primes of height $4$, so $\mathcal{J}_G$ is unmixed.

	However, $G\setminus\{5\}=K_{1,3}$, the claw. $\mathcal{J}_{K_{1,3}}$ is easily verified to be of K\"{o}nig type, but
	\begin{equation*}
		\mathcal{C}(K_{1,3})=\{\varnothing,\{1\}\}
	\end{equation*}
	which yields associated primes of heights $3$ and $2$, so $\mathcal{J}_{K_{1,3}}$ is not unmixed.

	Furthermore, $\mathcal{J}_G$ is not Cohen-Macaulay, since $\mathcal{C}(G)$ does not satisfy the conditions of \cite[Theorem 3.5]{bologniniCohenMacaulayBinomial2022}. This demonstrates that being unmixed of K\"{o}nig type is not sufficient to guarantee Cohen-Macaulayness (we have already seen that the converse does not hold either in \cref{example:net}).
\end{example}

Much research has been done on binomial edge ideals over fields of prime characteristic, particularly the study of their $F$-singularities. In the case of binomial edge ideals of K\"{o}nig type, LaClair and McCullough proved the following:
\begin{proposition}\thmCite[Proposition 10.3]{laclairFpurityBinomialEdge2026}
	Suppose that $k$ has prime characteristic $p>0$, and let $G$ be a graph. Then if $\mathcal{J}_G$ is unmixed of K\"{o}nig type, $R/\mathcal{J}_G$ is $F$-pure.
\end{proposition}

\subsection{Further Reading}

Further results on binomial edge ideals of K\"{o}nig type can be found in \cite[Section 3]{herzogGradedIdealsKonig2022}, \cite[Section 5]{laclairInvariantsBinomialEdge2025}, \cite{williamsLFCoversBinomialEdge2023}, and \cite[Section 10.1]{laclairFpurityBinomialEdge2026}.

\newpage

\section{A Graph-Theoretic Characterisation of Graphs with Binomial Edge Ideals of K\"{o}nig Type}\label{sec:mainResult}

\subsection{Path Covering \& Scattering Numbers}\label{subsec:PCSN}

We begin by defining several graph-theoretic invariants, and exploring some of their properties.

\begin{definition}\label{def:pi}
	Let $G$ be a graph. The \emph{path covering number} of $G$ is the minimum number of vertex-disjoint paths in $G$ required to cover its vertices. We denote this value by $\pi(G)$.
\end{definition}

\begin{note}
	Notation for this value varies across the literature. We follow that of \cite{deogun1ToughCocomparabilityGraphs1997}, $\pi_0$ is used in \cite{jungClassPosetsCorresponding1978}, and $\rho$ in \cite{giakoumakisScatteringNumberModular1997}. This list is not exhaustive, but addresses the most relevant papers for our purposes.
\end{note}

\begin{example}
	A graph $G$ is traceable if and only if $\pi(G)=1$.
\end{example}

\begin{proposition}\label{prop:LFPi}
	Let $G$ be a graph. Then
	\begin{equation*}
		\LF(G)=\smallAbs{V(G)}-\pi(G)
	\end{equation*}
\end{proposition}

\begin{proof}
	Let $F$ be a linear forest covering every vertex of $G$ (allowing isolated vertices as components), and let $m$ be the number of connected components of $F$ (that is, $c_F(\varnothing)$). Since every component $F_i$ of $F$ is a path, we know that
	\begin{equation*}
		\smallAbs{E(F_i)}=\smallAbs{V(F_i)}-1
	\end{equation*}
	Then, because $F$ covers every vertex of $G$, we have
	\begin{equation*}
		\smallAbs{E(F)}=\mySum_{i=1}^{m}\smallAbs{E(F_i)}=\mySum_{i=1}^{m}(\smallAbs{V(F_i)}-1)=\smallAbs{V(G)}-m
	\end{equation*}
	Now, $\LF(G)$ is the maximum value of $\smallAbs{E(F)}$ among all linear forests covering every vertex of $G$, and the minimum value of $m$ for such forests is $\pi(G)$. The result then follows.
\end{proof}

The following invariant was introduced by Jung in \cite[Section 4]{jungClassPosetsCorresponding1978}:

\begin{definition}\label{def:sc}
	Let $G$ be a graph. Then we define the \emph{scattering number} of $G$ as
	\begin{equation*}
		\scnum(G)\defeq\max\{\text{$c_G(S)-\smallAbs{S}:S\subseteq V(G)$ and $c_G(S)\neq1$}\}
	\end{equation*}
\end{definition}

\begin{note}
	Again, notation for this value varies. We follow that of \cite{deogun1ToughCocomparabilityGraphs1997}, whilst $s$ is used in \cite{jungClassPosetsCorresponding1978} and \cite{giakoumakisScatteringNumberModular1997}.
\end{note}

The scattering number has been described as a measure of the ``vulnerability'' of a network to disconnection (see, for example, \cite{kirlangicMeasureGraphVulnerability2002}). The lower the scattering number, the harder it is to disrupt the network by removing nodes. For example, the complete graph $K_n$ is maximally ``resilient'', with $\scnum(K_n)=-n$. The star $K_{1,n-1}$ however is maximally ``fragile'': we have $\scnum(K_{1,n-1})=n-2$ (the maximum value for a connected graph), since removing the central vertex completely disconnects the remainder of the graph.

The condition $c_G(S)\neq1$ is imposed to allow finer distinctions between connected graphs. Without it, $\scnum(G)$ would always be at least $1$ by taking $S=\varnothing$. However, for our purposes, removing this restriction is more natural.

\begin{definition}\label{def:scStar}
	Let $G$ be a graph. Then we define the \emph{unrestricted scattering number} of $G$ as
	\begin{equation*}
		\scstar(G)\defeq\max\{c_G(S)-\smallAbs{S}:S\subseteq V(G)\}
	\end{equation*}
	Equivalently (since $G$ has at least $1$ connected component, so $c_G(\varnothing)-\smallAbs{\varnothing}\geq1$),
	\begin{equation*}
		\scstar(G)\defeq\max\{1,\scnum(G)\}
	\end{equation*}
\end{definition}

\begin{note}
	This is not established terminology or notation.
\end{note}

\pagebreak

We do not need to consider all subsets of $V(G)$ to compute this value.
\begin{proposition}\label{prop:scSets}
	Let $G$ be a graph. Then
	\begin{equation*}
		\scstar(G)=\max\{c_G(S)-\smallAbs{S}:S\in\mathcal{C}(G)\}
	\end{equation*}
\end{proposition}

\begin{proof}
	Take any $S\subseteq V(G)$, and suppose that there exists some $v\in S$ such that adding $v$ back to $G\setminus S$ does not reconnect any components of $G\setminus S$. Then
	\begin{equation*}
		c_G(S\setminus\{v\})\geq c_G(S)
	\end{equation*}
	(with equality unless $v$ is isolated), and so
	\begin{equation*}
		c_G(S\setminus\{v\})-\smallAbs{S\setminus\{v\}}=c_G(S\setminus\{v\})-(\smallAbs{S}-1)\geq c_G(S)-\smallAbs{S}+1>c_G(S)-\smallAbs{S}
	\end{equation*}
	This tells us that, if $S$ maximises $c_G(S)-\smallAbs{S}$, then adding any vertex in $S$ back to $G\setminus S$ must reconnect some components of $G\setminus S$, and therefore $S\in\mathcal{C}(G)$.
\end{proof}

This allows us to relate $\hgt_R(\mathcal{J}_G)$ and $\scstar(G)$.
\begin{proposition}\label{prop:scHeight}
	Let $G$ be a graph. Then
	\begin{equation*}
		\hgt_R(\mathcal{J}_G)=\smallAbs{V(G)}-\scstar(G)
	\end{equation*}
\end{proposition}

\begin{proof}
	By \cref{cor:height} and \cref{prop:scSets}, we have
	\begin{equation*}
		\hgt_R(\mathcal{J}_G)=\min\{\smallAbs{V(G)}+\smallAbs{S}-c_G(S):S\in\mathcal{C}(G)\}=\smallAbs{V(G)}-\scstar(G)
	\end{equation*}
	as desired.
\end{proof}

Furthermore, we can relate $\pi$ and $\scstar$ via the following inequality:
\begin{proposition}\label{prop:scPiIneq}
	Let $G$ be a graph. Then
	\begin{equation*}
		\scstar(G)\leq\pi(G)
	\end{equation*}
\end{proposition}

\begin{proof}
	Let $F$ be a path covering of $G$ with $\pi(G)$ connected components, and take any $S\subseteq V(G)$. Since each component is a path, $F\setminus S$ has at most $\pi(G)+\smallAbs{S}$ connected components. Then, since $E(F)\subseteq E(G)$, we have
	\begin{equation*}
		c_G(S)\leq c_F(S)\leq\pi(G)+\smallAbs{S}
	\end{equation*}
	so
	\begin{equation*}
		c_G(S)-\smallAbs{S}\leq\pi(G)
	\end{equation*}
	for all $S\subseteq V(G)$, and the result follows.
\end{proof}

\subsection{The Characterisation}

\mainThm*

\begin{proof}\label{prf:main}
	By \cref{thm:heightLFKT}, \cref{prop:scHeight}, and \cref{prop:LFPi}, we have that $\mathcal{J}_G$ is of K\"{o}nig type if and only if
	\begin{equation*}
		\smallAbs{V(G)}-\scstar(G)=\smallAbs{V(G)}-\pi(G)
	\end{equation*}
	and the result is immediate.
\end{proof}

\begin{note}
	Since we always have $\scstar(G)\leq\pi(G)$ by \cref{prop:scPiIneq}, graphs with binomial edge ideals of K\"{o}nig type can be viewed as those for which this inequality is sharp.
\end{note}

We now have a characterisation of graphs with binomial ideals of K\"{o}nig type in purely graph-theoretic terms, echoing the case of (monomial) edge ideals. This allows us to apply known results from graph theory to the study of binomial ideals of K\"{o}nig type.

\newpage

\section{Classes of Graphs with Binomial Edge Ideals of K\"{o}nig Type}\label{sec:KTClasses}

\cref{thm:main} allows us to establish that several families of graphs have binomial edge ideals of K\"{o}nig type.

\subsection{Cocomparability/Weakly Closed Graphs}

\begin{definition}
	Let $G$ be a graph. An assignment of directions to the edges of $G$ is called an \emph{orientation}. An orientation is said to be \emph{transitive} if, whenever directed edges $(u,v)$ and $(v,w)$ exist, the directed edge $(u,w)$ exists also.

	If a transitive orientation of $G$ exists, we say that $G$ is a \emph{comparability} graph.
\end{definition}

\begin{definition}\label{def:cocomp}
	We say that a graph $G$ is a \emph{cocomparability} graph if its complement $\overline{G}$ is comparability. 
\end{definition}

\begin{proposition}\thmCite[Theorem 8]{deogun1ToughCocomparabilityGraphs1997}\label{prop:cocomEq}
	For any cocomparability graph $G$, we have
	\begin{equation*}
		\pi(G)=\max\{1,\scnum(G)\}
	\end{equation*}
\end{proposition}

\begin{note}
	Deogun, Kratsch and Steiner attribute an earlier proof of \cref{prop:cocomEq} to an unpublished  manuscript \cite{lehelPathPartitionCocomparability1991} of Lehel.
\end{note}

\begin{theorem}\label{thm:cocomKT}
	For any cocomparability graph $G$, $\mathcal{J}_G$ is of K\"{o}nig type.
\end{theorem}

\begin{proof}
	This follows immediately from \cref{thm:main} and \cref{prop:cocomEq} (using our $\scstar$ notation).
\end{proof}

Since all cographs are permutation graphs, and both permutation and interval graphs are subclasses of cocomparability graphs (see \cite{brandstadtGraphClassesSurvey1999}), this proves the three conjectures of \cite[Section 2.2]{williamsLFCoversBinomialEdge2023}.

This class also includes complete bipartite and trivially perfect graphs (again, see \cite{brandstadtGraphClassesSurvey1999}), and so generalises \cite[Proposition 2.21]{williamsLFCoversBinomialEdge2023} and \cite[Proposition 2.23]{williamsLFCoversBinomialEdge2023}.

Within the field of binomial edge ideals, cocomparability graphs are better known by terminology introduced by Matsuda:
\begin{definition}\thmCite[Definition 2.1]{matsudaWeaklyClosedGraphs2018}\label{def:WC}
	Let $G$ be a graph on $n$ vertices. We say that $G$ is \emph{weakly closed} if there exists a labelling of the vertices of $G$ with the following property: for all integers $1\leq i<j<k\leq n$, if $\{i,k\}\in E(G)$, then $\{i,j\}\in E(G)$ or $\{j,k\}\in E(G)$.
\end{definition}

\begin{note}
	This concept generalises the \textit{closed} graphs of \cite[p.~319]{herzogBinomialEdgeIdeals2010}.
\end{note}

\begin{theorem}\thmCite[Theorem 2.9]{matsudaWeaklyClosedGraphs2018}\label{thm:WCCocom}
	Let $G$ be a graph. Then $G$ is weakly closed if and only if $G$ is cocomparability.
\end{theorem}

\begin{note}
	Corneil, Olariu and Stewart remark in \cite[p.~223]{corneilAsteroidalTriplefreeGraphs1994} that \cref{thm:WCCocom} was already known to Kratsch and Stewart in \cite{kratschDominationCocomparabilityGraphs1993}. 
\end{note}

In \cite[Question 10.5]{laclairFpurityBinomialEdge2026}, LaClair and McCullough raised the question of whether binomial edge ideals of weakly closed graphs are of K\"{o}nig type under the additional assumption of unmixedness. We can now answer this in the affirmative, and in fact without requiring the unmixedness hypothesis.
\wcKTCor*

\begin{proof}\label{prf:WCKT}
	This follows immediately from \cref{thm:cocomKT} and \cref{thm:WCCocom}.
\end{proof}

\subsection{\texorpdfstring{\protect $P_4$-Extendible Graphs}{P₄-Extendible Graphs}}\label{OtherKTClasses}

\begin{definition}
	We say that $G$ is \emph{$P_4$-extendible} if, for every induced subgraph $H$ of $G$ isomorphic to $P_4$, there is at most one vertex in $V(G)\setminus V(H)$ which belongs to an induced $P_4$ with three vertices of $H$.
\end{definition}

\begin{proposition}
	For any $P_4$-extendible graph $G$, $\mathcal{J}_G$ is of K\"{o}nig type.
\end{proposition}

\begin{proof}
	This follows immediately from \cref{thm:main} and \cite[Theorem 2]{hochstattlerHamiltonicityGraphsFewP1995}.
\end{proof}

\pagebreak

\subsection{\texorpdfstring{\protect $3$-Sun-Free Semi-$P_4$-Sparse Graphs}{3-Sun-Free Semi-P₄-Sparse Graphs}}

\begin{definition}
	We say that a graph is \emph{$3$-sun-free semi-$P_4$-sparse} if it contains no induced $P_5$, $\overline{P_5}$ (otherwise known as the house), kite, or $3$-sun (otherwise known as the net).
\end{definition}

\begin{note}
	We use the term \textit{$3$-sun} in agreement with \cite[p.~323]{giakoumakisScatteringNumberModular1997}, but this construction is usually referred to instead as the \textit{$3$-sunlet}.
\end{note}

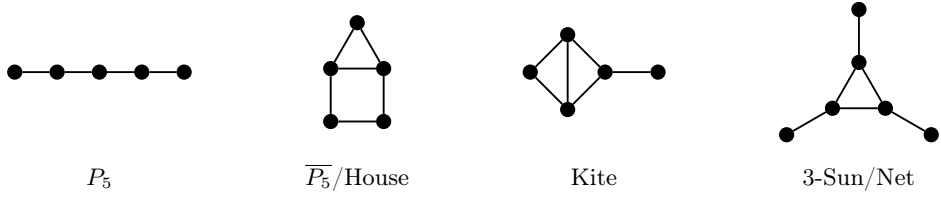
\begin{center}
	\begin{tabular}{@{}c@{\hspace{1.5cm}}c@{\hspace{1.5cm}}c@{\hspace{1.5cm}}c@{}}
		\begin{tikzpicture}[mygraph,x=0.7cm,y=0.7cm,baseline={([yshift=-0.5ex]current bounding box.center)}]
       		\node(1) at (0,0) {};
    		\node(2) at ([shift=(0:0.8)]1) {};
        	\node(3) at ([shift=(0:0.8)]2) {};
			\node(4) at ([shift=(0:0.8)]3) {};
			\node(5) at ([shift=(0:0.8)]4) {};
	  	  	\foreach \from/\to in {1/2,2/3,3/4,4/5}
				\draw[-] (\from) -- (\to);
    	\end{tikzpicture}
		&
		\begin{tikzpicture}[mygraph,x=0.7cm,y=0.7cm,baseline={([yshift=-0.5ex]current bounding box.center)}]
       		\node(1) at (0,0) {};
    		\node(2) at ([shift=(90:1)]1) {};
        	\node(3) at ([shift=(0:1)]1) {};
			\node(4) at ([shift=(0:1)]2) {};
			\node(5) at ([shift=(60:1)]2) {};
	  	  	\foreach \from/\to in {1/2,1/3,2/4,2/5,3/4,4/5}
				\draw[-] (\from) -- (\to);
    	\end{tikzpicture}
		&
		\begin{tikzpicture}[mygraph,x=0.7cm,y=0.7cm,baseline={([yshift=-0.5ex]current bounding box.center)}]
       		\node(1) at (0,0) {};
    		\node(2) at ([shift=(135:1)]1) {};
        	\node(3) at ([shift=(45:1)]1) {};
			\node(4) at ([shift=(45:1)]2) {};
			\node(5) at ([shift=(0:1)]3) {};
	  	  	\foreach \from/\to in {1/2,1/3,1/4,2/4,3/4,3/5}
				\draw[-] (\from) -- (\to);
    	\end{tikzpicture}
		&
		\begin{tikzpicture}[mygraph,x=0.7cm,y=0.7cm,baseline={([yshift=-0.5ex]current bounding box.center)}]
       		\node(1) at (0,0) {};
    		\node(2) at ([shift=(-120:1)]1) {};
        	\node(3) at ([shift=(0:1)]2) {};
			\node(4) at ([shift=(90:1)]1) {};
			\node(5) at ([shift=(-150:1)]2) {};
        	\node(6) at ([shift=(-30:1)]3) {};
	  	  	\foreach \from/\to in {1/2,1/3,1/4,2/3,2/5,3/6}
				\draw[-] (\from) -- (\to);
    	\end{tikzpicture}
		\\[1cm]
		\small $P_5$
		&
		\small $\overline{P_5}$/House
		&
		\small Kite
		&
		\small $3$-Sun/Net
	\end{tabular}
\end{center}

This class evolved to generalise the work of \cite{hochstattlerHamiltonicityGraphsFewP1995} on $P_4$-sparse graphs (see \cite[Definition 3]{hochstattlerHamiltonicityGraphsFewP1995}), which in turn generalised the work of \cite{jungClassPosetsCorresponding1978} on cographs (there called $D^\ast$-graphs, using older terminology).

\begin{proposition}
	For any $3$-sun-free semi-$P_4$-sparse graph $G$, $\mathcal{J}_G$ is of K\"{o}nig type.
\end{proposition}

\begin{proof}
	This follows immediately from \cref{thm:main} and \cite[Section 4, Theorem 7]{giakoumakisScatteringNumberModular1997}.
\end{proof}

\subsection{Jung Graphs}

\cite{giakoumakisScatteringNumberModular1997} introduces a class of graphs which, by definition, satisfy the equality of \cref{thm:main}:

\begin{definition}\thmCite[Section 2, Definition 1]{giakoumakisScatteringNumberModular1997}
	We say that a graph $G$ is a \emph{Jung graph} if it satisfies the following conditions:
	\begin{enumerate}
		\item $\pi(G)=\max\{1,\scnum(G)\}$ (which is $\scstar(G)$ in our notation).
		\item If $\scnum(G)=0$, then $G$ is Hamiltonian.
		\item If $\scnum(G)<0$, then $G$ is Hamiltonian-connected.
	\end{enumerate}
\end{definition}
Therefore, by \cref{thm:main}, all Jung graphs have binomial edge ideals of K\"{o}nig type.

\cite[Section 4, Theorem 7]{giakoumakisScatteringNumberModular1997} in fact shows that $3$-sun-free semi-$P_4$-sparse graphs are Jung graphs.

There does not appear to be much literature on Jung graphs, but \cite{giakoumakisP_4tidyGraphs1997} and \cite{liSpanningConnectednessHamiltonian2015} demonstrate some families for which these conditions hold.

\subsection{Previously Known Classes}

\subsubsection{Traceable Graphs}

It is well known that traceable graphs have binomial edge ideals of K\"{o}nig type (see \cref{example:traceable}). \cref{thm:main} offers a short alternative proof:

\begin{proposition}\label{prop:traceableKT}
	For any traceable graph $G$, $\mathcal{J}_G$ is of K\"{o}nig type.
\end{proposition}

\begin{proof}
	Since $G$ is traceable, we know $\pi(G)=1$, so by \cref{prop:scPiIneq} we have
	\begin{equation*}
		1\leq\scstar(G)\leq\pi(G)=1
	\end{equation*}
	and we are done by \cref{thm:main}.
\end{proof}

\subsubsection{Forests}

In \cite[Theorem 5.9]{laclairInvariantsBinomialEdge2025}, LaClair proved that forests have binomial edge ideals of K\"{o}nig type via linear programs. The author gave an alternative proof in \cite[Theorem 2.33]{williamsLFCoversBinomialEdge2023}. A result of Jung yields another proof:

\begin{proposition}\label{prop:forestKT}
	For any forest $F$, $\mathcal{J}_F$ is of K\"{o}nig type.
\end{proposition}

\begin{proof}
	This follows immediately from \cref{thm:main} and \cite[Theorem 4.1]{jungClassPosetsCorresponding1978}.
\end{proof}

\newpage

\section{Unmixed Binomial Edge Ideals of K\"{o}nig Type}\label{sec:umKT}

We will now present an alternative classification of unmixed binomial edge ideals of K\"{o}nig type to that of Herzog, Hibi and Moradi (\cref{prop:HHMChar}). This builds on the following result of LaClair and McCullough:
\begin{theorem}\thmCite[Corollary 10.4]{laclairFpurityBinomialEdge2026}\label{thm:LM}
	Let $G$ be a graph, and suppose that $\mathcal{J}_G$ is unmixed of K\"{o}nig type. Then $G$ is weakly closed.
\end{theorem}

This allows us to deduce our classification:
\unmixedCor*

\begin{proof}\label{prf:unmixed}
	This follows immediately from \cref{cor:WCKT} and \cref{thm:LM}.
\end{proof}

\begin{note}
	Just as we \hyperref[note:C4]{noted} following \cref{prop:HHMChar}, \cref{cor:unmixed} does \textit{not} say that $\mathcal{J}_G$ is unmixed of K\"{o}nig type if and only if $G$ is weakly closed. The same counterexample of $C_4$ demonstrates this, it is easily seen to be weakly closed given that same labelling.
\end{note}

As discussed in the \hyperref[sec:intro]{introduction}, \cref{cor:unmixed} has an algorithmic advantage over \cref{prop:HHMChar}. Checking whether a graph is traceable amounts to deciding if it admits a Hamiltonian path. \cite[p.~60]{gareyComputersIntractabilityGuide1979}, for example, demonstrates that this is NP-complete in general, and so the problem is computationally difficult. By contrast, \cite[Theorem 5.33]{golumbicAlgorithmicGraphTheory2004} shows that comparability graphs can be recognised in polynomial time. Since cocomparability graphs are, by definition, the complements of comparability graphs, these can be recognised in polynomial time also.

Combining the two characterisations yields a result which, on the face of it, appears to be totally unrelated to ideals of K\"{o}nig type.
\begin{corollary}\label{cor:wcTraceable}
	Let $G$ be a graph, and suppose that $\mathcal{J}_G$ is unmixed. Then $G$ is weakly closed if and only if every connected component of $G$ is traceable.
\end{corollary}

\begin{proof}
	This follows immediately from \cref{cor:unmixed} and \cref{prop:HHMChar}.
\end{proof}

\begin{note}\label{note:C5}
	The unmixedness hypothesis is necessary in \cref{cor:wcTraceable}. For example, clearly $C_5$ is traceable, however its complement $\overline{C_5}\cong C_5$ is not a comparability graph: we cannot orient edges $u-v-w$ as $u\to v\to w$, since no edge $u-w$ exists, and so the orientation of each edge must alternate, which is impossible for a cycle of odd length. Then $C_5$ is not weakly closed.
\end{note}

It was \hyperref[note:almostAllKT]{noted} in \cref{subsec:KTExamples} that almost all graphs have binomial edge ideals of K\"{o}nig type. In the unmixed case, quite the opposite is true. This follows easily from the following lemma:
\begin{lemma}\label{lem:almostAllH}
	For any graph $H$, almost every graph contains an induced subgraph isomorphic to $H$. That is, letting $\mathcal{N}(n)$ denote the number of (unlabelled) graphs on $n$ vertices, and $\mathcal{H}(n)$ the number of (unlabelled) graphs on $n$ vertices containing an induced subgraph isomorphic to $H$, we have
	\begin{equation*}
		\lim_{n\to\infty}\frac{\mathcal{H}(n)}{\mathcal{N}(n)}=1
	\end{equation*}
\end{lemma}

\begin{proof}
	Applying \cite[Proposition 11.3.1]{diestelGraphTheory2017}, taking $p=\frac{1}{2}$ to obtain the uniform distribution on labelled graphs, proves the claim in the labelled case. We then immediately obtain the result for unlabelled graphs via \cite[Section 9.4, Metatheorem]{hararyGraphicalEnumeration1973}.
\end{proof}

\begin{proposition}\label{prop:almostNoUMKT}
	Almost no graphs have unmixed binomial edge ideals of K\"{o}nig type.
\end{proposition}

\begin{proof}
	By \cite[Proposition 10.7]{laclairFpurityBinomialEdge2026}, the class of weakly closed graphs is closed under vertex deletion. We saw in the \hyperref[note:C5]{note} following \cref{cor:wcTraceable} that $C_5$ is not weakly closed, and so no weakly closed graph can contain $C_5$ as an induced subgraph. Then, by \cref{lem:almostAllH}, almost no graphs are weakly closed, and we are done by \cref{cor:unmixed}.
\end{proof}

This illustrates another aspect in which \cref{cor:unmixed} can be advantageous compared to \cref{prop:HHMChar}: traceable graphs are asymptotically common, whereas weakly closed graphs are asymptotically rare.

\pagebreak

\section{A Conjecture for AT-Free Graphs}\label{sec:ATFreeConj}

\begin{definition}
	Let $G$ be a graph. We say that three pairwise non-adjacent vertices of $G$ form an \emph{asteroidal triple} if, for each pair of these vertices, there exists a path connecting them which avoids the neighbourhood of the third.
\end{definition}

\begin{definition}\label{def:ATFree}
	We say that a graph $G$ is \emph{AT-free} if it contains no asteroidal triple.
\end{definition}

Binomial edge ideals of AT-free graphs have already appeared in the literature, for example \cite[Corollary 4.13]{laclairFpurityBinomialEdge2026}. We make the following conjecture:

\ATFreeConj*

\label{MacaulayEvidence}We have verified this conjecture using \citeMacaulay{} for all connected AT-free graphs with up to 9 vertices, a total of 95{,}869 graphs. The conjecture then holds for \textit{all} AT-free graphs with up to $9$ vertices.

Cocomparability graphs comprise a notable subclass of AT-free graphs (\cite[Theorem 4]{golumbicToleranceGraphs1984}), and so \cref{thm:cocomKT} lends further weight to \cref{conj:ATFreeKT}.

\label{ATFreeDiscussion}We do not expect to find a proof of this conjecture via a straightforward algorithm such as \cite[Algorithm 2.34]{williamsLFCoversBinomialEdge2023}. This algorithm showed that forests have binomial edge ideals of K\"{o}nig type by constructing an $\LF$-cover (see \cite[Definition 2.4]{williamsLFCoversBinomialEdge2023}) in polynomial time, which also computes $\LF(F)$ for any forest $F$.

Since a graph $G$ is traceable if and only if $\LF(G)=\smallAbs{V(G)}-1$, a similar algorithm in the case of AT-free graphs would, as a special case, settle a long-standing problem in graph theory: whether the Hamiltonian path problem for AT-free graphs is solvable in polynomial time. This problem was noted to be open as late as 2023 (\cite[p.~2]{gomezPathEccentricityGraphs2023}), and does not appear to have been resolved since.

The Hamiltonian path problem has, however, been shown to be solvable in quadratic time for cocomparability graphs (\cite[Corollary 1]{damaschkeFindingHamiltonianPaths1992}). We do not claim then that such an algorithm is impossible, but rather that it would likely require substantial graph-theoretic content.

\begin{note}
	Cocomparability graphs are a subclass of both AT-free graphs and perfect graphs (see \cite{brandstadtGraphClassesSurvey1999}), and there are many perfect graphs which are not AT-free (and vice versa). For example, $C_6$ is not AT-free since any three pairwise non-adjacent vertices form an asteroidal triple, but it is perfect since it is bipartite (again, see \cite{brandstadtGraphClassesSurvey1999}).

	Since it is traceable, it has a binomial edge ideal of K\"{o}nig type, and so one might be tempted to conjecture that all perfect graphs have binomial edge ideals of K\"{o}nig type as an alternative strengthening of \cref{thm:cocomKT}. This is, however, not the case.

	We saw in \cref{example:bipCE} that $\mathcal{J}_G$ is not of K\"{o}nig type for
	\begin{align*}
		G&=\quad
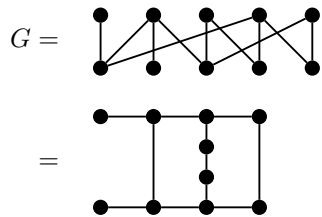
\begin{tikzpicture}[mygraph,x=0.7cm,y=0.7cm,baseline={([yshift=-0.5ex]current bounding box.center)}]
      		\node(1) at (0,0) {};
			\node(2) at (1,0) {};
			\node(3) at (2,0) {};
			\node(4) at (3,0) {};
			\node(5) at (4,0) {};
			\node(6) at (0,1) {};
			\node(7) at (1,1) {};
			\node(8) at (2,1) {};
			\node(9) at (3,1) {};
			\node(10) at (4,1) {};
	   		\foreach \from/\to in {1/6,1/7,1/9,2/7,3/7,3/8,3/10,4/8,4/9,5/9,5/10}
       	 		\draw[-] (\from) -- (\to);
    	\end{tikzpicture}\\[0.3cm]
		&=\quad\begin{tikzpicture}[mygraph,x=0.7cm,y=0.4cm,baseline={([yshift=-0.5ex]current bounding box.center)}]
	       	\node(1) at (1,3) {};
			\node(2) at (0,0) {};
			\node(3) at (2,0) {};
			\node(4) at (2,2) {};
			\node(5) at (3,3) {};
			\node(6) at (0,3) {};
			\node(7) at (1,0) {};
			\node(8) at (2,1) {};
			\node(9) at (2,3) {};
			\node(10) at (3,0) {};
		    \foreach \from/\to in {1/6,1/7,1/9,2/7,3/7,3/8,3/10,4/8,4/9,5/9,5/10}
	   	     	\draw[-] (\from) -- (\to);
		\end{tikzpicture}
	\end{align*}
	As we just noted, bipartite graphs are also perfect, and so this alternative strengthening then fails.

	However, it is easily seen that $G$ is not AT-free (for example, the two vertices of degree 1 together with any of the vertices of degree 2 here form an asteroidal triple), and so this does not affect \cref{conj:ATFreeKT}.
\end{note}

\section*{Acknowledgements}

The author is grateful to Dr Adam LaClair for his helpful comments on an \href{https://arxiv.org/abs/2605.24933v1}{earlier version} of this paper.

Furthermore, the author wishes to acknowledge the use of generative AI in exploring works related to binomial edge ideals of K\"{o}nig type. In particular, the author thanks OpenAI's GPT-5.5 for suggesting investigating path covering and scattering numbers, and Anthropic's Claude Opus 4.7 for bringing \cite{deogun1ToughCocomparabilityGraphs1997} to their attention. All mathematical content and prose is the author's own.

\newpage

\nocite{M2}
\printbibliography

\end{document}